\documentclass{amsart}
\usepackage{amsfonts}
\usepackage{graphicx}
\usepackage{amscd}
\usepackage{amsmath}
\usepackage{amssymb}
\usepackage{stmaryrd} 

\makeatletter
\@namedef{subjclassname@2010}{%
  \textup{2010} Mathematics Subject Classification}
\makeatother

\theoremstyle{plain}
\newtheorem*{acknowledgement}{Acknowledgement}

\newtheorem{corollary}{\bf Corollary}
\newtheorem{definition}{\bf Definition}
\newtheorem{lemma}{\bf Lemma}

\newtheorem{proposition}{\bf Proposition}
\newtheorem{remark}{Remark}

\newtheorem{theorem}{\bf Theorem}
\newcommand{\Limsup}{\operatorname*{Lim\,sup}}

\theoremstyle{definition}

\numberwithin{equation}{section}

\title[Article Title]{A Refined Proximal Algorithm for Nonconvex Multiobjective Optimization in Hilbert Spaces}

\author{G. C. Bento$^\ast$}

\author{J. X. Cruz Neto}

\author{J. O. Lopes}

\author{B. S. Mordukhovich}

\author{P. R. Silva Filho}

\address[G. C. Bento]{Institute of Mathematics and Statistics, IME, Federal University of Goi\'as, Goi\^ania, Goi\'as, Brazil}
\email{glaydston@ufg.br}

\address[J. X. Cruz Neto]{Departamento de Matem\'{a}tica, CCN, Federal University of Piau\'{\i}\\
 Te\-re\-si\-na, Piau\'{\i}, Brazil.}
\email{jxavier@ufpi.edu.br}

\address[J. O. Lopes]{Departamento de Matem\'{a}tica, CCN, Federal University of Piau\'{\i}\\
 Te\-re\-si\-na, Piau\'{\i}, Brazil.}
\email{jurandir@ufpi.edu.br}

\address[B. S. Mordukhovich]{Department of Mathematics, Wayne State University, Detroit, Michigan, United States}
\email{boris@math.wayne.edu}

\address[P. R. Silva Filho]{Departamento de Matem\'{a}tica, CCN, Federal University of Piau\'{\i}\\
 Te\-re\-si\-na, Piau\'{\i}, Brazil.}
\email{pedrorodrigues@ufpi.edu.br}


\begin{document}

\newcommand{\spacing}[1]{\renewcommand{\baselinestretch}{#1}\large\normalsize}
\spacing{1.2}

\begin{abstract}
This paper is devoted to general nonconvex problems of multiobjective optimization in Hilbert spaces. Based on Mordukhovich's limiting subgradients, we define a new notion of Pareto critical points for such problems, establish necessary optimality conditions for them, and then employ these conditions to develop a refined version of the vectorial proximal point algorithm with providing its detailed convergence analysis. The obtained results largely extend those initiated by Bonnel, Iusem and Svaiter \cite{Bonnel2005} for convex vector optimization problems and by Bento et al. \cite{Bento2018}  for nonconvex finite-dimensional problems in terms of Clarke's generalized gradients.
\end{abstract}

\maketitle

\section{Introduction}\label{sec1}

Multiobjective optimization involves optimizing two or more real-valued objective functions simultaneously. Typically, no single point can minimize all the given objective functions simultaneously, which leads us to the concept of Pareto optimality replacing the traditional notion of optimality.  Restricting our discussion to the {\em proximal point algorithm} (PPA), the focus of this paper is to improve the (exact) method introduced and analyzed by Bonnel, Iusem sand Svaiter \cite{Bonnel2005}. In the mentioned paper, the authors employ the method, in the fully convex setting, to find a weak Pareto minimizer of a vector function $F:X\to Y$ from a real Hilbert space $X$ to a real Banach space $Y$ containing a closed, convex, and pointed cone $\mathcal{C}$ with nonempty interior. The method constructs as $(k+1)$th iteration a weak Pareto solution of $F_k(x):=F(x)+\lambda_k||x-x^k||^2 \varepsilon^k$
constrained to the set $\Omega_k:=\{ x\in X\; :\; F(x)\preceq_{\mathcal{C}} F(x^k)\}$. Here $\{\lambda_k\}$ represents a bounded sequence of positive scalars, and $\varepsilon^k$ denotes an exogenously selected vector from the interior of $\mathcal{C}$ with $||\varepsilon^k||=1$ for each $k=0,1,\ldots$. The strategy used in the convergence results is based on the first-order optimality condition for the following scalar problem:
\begin{equation}\label{scalarizedapproach}
\min_{x \in \Omega_k}\eta_k(x),
\end{equation}
where $\eta_k(x):=\langle F(x),z^k\rangle+\frac{\lambda_k}{2}\langle \varepsilon^k,z^k\rangle ||x-x^k||^2$ and $\{z^k\}\in \mathcal{C}^{+}\subset Y^*$, $||z^k||=1$, is an exogenous sequence,  $\mathcal{C}^{+}$ denotes the positive polar cone and $Y^*$ is the topological dual space of $Y$ with $\langle \cdot,\cdot\rangle:Y\times Y^*\to \mathbb{R}$ being the duality pairing. Therefore, solving \eqref{scalarizedapproach} implies that 
\begin{equation}\label{scalarizedoptimincondition}
0 \in \partial \psi_k(x^{k+1})+\lambda_k\langle \varepsilon^k,z^k\rangle(x^{k+1}-x^k),
\end{equation}
where $\psi_k(x):=\langle F(x),z^k\rangle+\delta_{\Omega_k}(x)$, where $\partial \psi_k$ denotes the subdifferential of $\psi_k(\cdot)$ in the sense of convex analysis, and where $\delta_{\Omega_k}(\cdot)$ is the indicator function of $\Omega_k$. Bonnel, Iudsem and Svaiter \cite{Bonnel2005} prove that any sequence generated by the PPA converges (in the weak topology of $X$) to a weakly efficient minimizer of $F$, provided that a certain completeness condition of $Y$ is satisfied. Note that \eqref{scalarizedoptimincondition} can be viewed as
\begin{equation}
\nonumber\alpha_{k}(x^k-x^{k+1})\in \partial (\langle F(\cdot),z^k\rangle)(x^{k+1})+N(x^{k+1};\Omega_k),
\end{equation}
where $\alpha_k=\lambda_k\langle \varepsilon^k,z^k\rangle$ and $N(x^{k+1};\Omega_k)$ stands for the normal cone to $\Omega_k$ at $x^{k+1}\in \Omega_k$ in the classical sense of convex analysis. In this approach, the set $\Omega_k$ ensures that the algorithm is functioning as a descent process. 

A motivation, in a dynamic context, to consider the constrained set $\Omega_k$ is provided by Bento et al.~\cite{Bentocruzneto2014}. They observe that the set $\Omega_k$ delineates a vectorial improvement process, where a vector minimizing solution $x^k$ of the current proximal problem moves to a next one in a manner that it improves the current solution. This aspect is crucial for justifying the process at a behavioral level, where a risk-averse agent agrees to make changes only if they lead us to improvements across all aspects (all components of the vector). In the case of a group of agents, as considered here, this constraint set becomes even more significant. It ensures that the payoff of each agent in the group does not diminish. For variations of the algorithm initially introduced  in \cite{Bonnel2005} for convex vector problems see, for example, Ceng and Yao~\cite{Ceng2007}, Choung et al.~\cite{Choung2011}, Greg\'orio and Oliveira~\cite{Ronaldopaulo} and Villacorta and Oliveira~\cite{Kellypaulo}. Discussions about the $\mathbb{R}^m_+$-quasiconvex case have been explored in the works of Bento et al.\cite{Bentocruzneto2014} and Apolinario et al.\cite{Helena}. In these studies, their respective algorithms compute a point $x^{k+1}$ at the $(k+1)$th iteration satisfying the equation:
\begin{equation}
\nonumber 0 \in \partial g(F(x^{k+1}))+\alpha_{k}(x^{k+1}-x^{k})+N(x^{k+1};\Omega_k),
\end{equation}
where $g:\mathbb{R}^m\to \mathbb{R}$ is a kind of ``scalarization" function, where $\partial g$ denotes some subdifferential of $g$, and where $\{\alpha_k\}$ is a sequence of positive numbers. For an alternative approach that does not necessarily depend on the convexity of $\Omega_k$ (the previously mentioned works depend on the convexity of $\Omega_k$), we refer to Bento et al.~\cite{Bento2018}. In the aforementioned paper, the authors extend the applicability of the PPA (in finite dimensions) to locally Lipschitz vector-valued functions. It is worth pointing out that the key aspect of this alternative approach lies in selecting each iteration of the algorithm as a weak Pareto solution obtained through an optimality condition directly applied to the vector subproblem restricted to $\Omega_k$; namely, extended conditions of the Fritz-John type
given by Clarke's generalized gradients, which are necessary for weak Pareto-optimal solutions to multiobjective problems.

The aim of this paper is to investigate generalized conditions of the Fritz-John type in the fully {\em nonconvex} setting, characterized by Mordukhovich's {\em limiting subgradients}, and to broaden the applicability of the PPA (in {\em Hilbert spaces}) to locally Lipschitz vector-valued functions. More specifically, we aim to ensure that each weak cluster point represents a {\em critical Pareto point} of the problem defined via limiting subgradients. Additionally, under the supplementary assumptions of {\em  pseudoconvexity} of the vector-valued function and completeness of its counter-domain, we demonstrate that thge entire sequence of iterates generated by the PPA converges (in the weak topology of the space) to a Pareto critical point of the problem. 

The rest of the paper is organized as follows. Section~\ref{secMultiobectiveoptmin} introduces major notations together with some concepts and results of multiobjective optimization  and generalized differentiation in variational analysis needed below. In Section~\ref{algorithm}, we define the proximal point method to solve nonconvex multiobjective optimization problems and then establish its important properties. Finally, Section~\ref{sec:4} is devoted to the convergence analysis of the algorithm. 
\section{Multiobjective optimization and generalized differentiation}  \label{secMultiobectiveoptmin}

In this section, we revisit fundamental definitions and properties in multiobjective optimization as outlined, for example, in the book by Luc~\cite{DinhLuc1989}. Furthermore, some notions and results of variational analysis and generalized differentiation, taken mostly from the books by Mordukhovich \cite{Morduk} and Rockafellar-Wets \cite{rw} and playing a crucial role in the main achievements of this paper, are also presented here. 

We consider the $m$-dimensional Euclidean space $\mathbb{R}^m$ with the partial order denoted by $\preceq$ induced by the Pareto cone $\mathbb{R}^m_{+}$ as follows: $y\preceq z$ (or $z\succeq y$) if and only if $z-y\in\mathbb{R}^{m}_{+}$. The associated notion $y\prec Z$ means that $z-y\in\mathbb{R}^{m}_{++}$, where $\mathbb{R}^{m}_{+}$ and $\mathbb{R}^{m}_{++}$ are defined by
$$
\mathbb{R}^{m}_{+}:=\left\{x\in\mathbb{R}^m~:~x_j\geq0, ~ j\in \mathcal{I}\right\}, \quad \mathbb{R}^{m}_{++}:=\left\{x\in\mathbb{R}^m~:~x_j>0, ~j\in \mathcal{I}\right\},
$$
and where $\mathcal{I}:=\{1,\ldots,m\}$. Let $\mathcal{H}$ be a real Hilbert space with the inner product and the norm denoted, respectively, by $\langle \cdot, \cdot \rangle$ and $\|\cdot\|$. Given a nonempty closed set $\Omega\subset \mathcal{H}$ and a vector-valued function $F=(f_{1},\ldots,f_{m}):\mathcal{H}\to \mathbb{R}^m$, the problem of finding a {\em weak Pareto point} of $F$ in $\Omega$ consists of determining a point $x^*\in \Omega$ such that there exists no vector $x\in \Omega$ with $F(x)\prec F(x^*)$. We denote this problem by
\begin{equation}\label{eq:mp}
\mbox{min}_w \{F(x)~:~x\in \Omega\}.
\end{equation}
By $\mbox{arg}\min_{w}\{ F(x): x \in \Omega\}$ denote the collection of all weak Pareto points of the vector  function $F$ in the set $\Omega$.  

\begin{remark}\label{exp}
As mentioned in Huang and Yang~\cite{huang}, the vector functions 
\[
F(\cdot)\quad \mbox{and}\quad e^{F(\cdot)}:=(e^{f_{1}(\cdot)},\ldots,e^{f_{m}(\cdot)})
\]
have the same sets of weak Pareto points, where $e^{\alpha}$ denotes the exponential map valued at $\alpha\in\mathbb{R}$. This result can be easily extended to Pareto critical setting in Definition~\ref{crit-pareto}. Therefore, concerning Pareto critical points, we can assume without loss of generality that $F\succ 0$.
\end{remark}

Let $f:\mathcal{H}\to\mathbb{R}$ be a lower semicontinuous function.
For a given sequence $\{u_k\}\subset \mathcal{H}$, we use the notation $u^k\stackrel{f}{\to} \bar{u}$ to signify that $u^k\to \bar{u}$ with $f(u^k)\to f(\bar{u})$. As  mentioned in the introduction, the objective function is not necessarily differentiable. To handle such situations, we introduce the following concepts. The {\em regular/Fr\'echet subdifferential}, of $f$ at $\bar{a}\in\mathcal{H}$ is defined by
\begin{equation*}
\partial^F f(\overline{a})=\bigg\{v\in \mathcal{H}:\,\liminf_{a\to\overline{a}}\dfrac{f(a)-f(\overline{a})-\langle v,a-\overline{a}\rangle}{\|a-\overline{a}\|}\geq 0\bigg\}.
\end{equation*}
It can be equivalently described as is the set of vectors $v\in \mathcal{H}$ such that
\begin{align}
\nonumber f(a)\geq f(\bar{a})+\langle v,a-\overline{a}\rangle+o(\|a-\overline{a}\|),
\end{align}
where $\displaystyle\lim_{a\to\overline{a}}\dfrac{o(\|a-\overline{a}\|)}{\|a-\overline{a}\|}=0$. The {\em limiting/Mordukhovich subdifferential} of $f$ at $\bar{a}\in\mathcal{H}$ is defined by
\begin{equation*}
\partial^Mf(\bar{a})=\Limsup_{a\stackrel{f}{\to} \bar{a}}\partial^Ff(a)=\left\{v\in \mathcal{H}:\exists\, a^k\stackrel{f}{\to} \bar{a} ,\,v^k\in\partial^Ff(a^k),\,v^k\rightharpoonup v\right\},
\end{equation*}
 where ``Limsup" signifies the outer limit of set-valued mappings in the norm-weak topology of ${\mathcal H}\times{\mathcal H}$.

It is well known that if $f$ is differentiable at $\bar{a} \in \mathcal{H}$, then its Fr\'echet subdifferential is a singleton, $\partial^Ff(\bar{a}) = \{\nabla f(\bar x)\}$.
If $f$ is continuously differentiable around $\bar{a} \in \mathcal{H}$ (or merely strictly differentiable at this point), then its Mordukhovich subdifferential reduces to the same singleton.

Let $f:\mathcal{H}\to \mathbb{R}$ be a locally Lipschitz function around $x\in\mathcal{H}$, and let $d\in\mathcal{H}$. The (Clarke) {\em generalized directional derivative} of $f$ at $x$ in the direction of $d$, denoted by $f^{\circ}(x, d)$, is defined by
\begin{align}
\nonumber f^{\circ}(x,d)=\limsup_{{y \to x} \atop {t \downarrow 0 }} \dfrac{f(y+td)-f(y)}{t}.
\end{align}
The {\em Clarke subdifferential/generalized gradient} of $f$ at $x$, denoted by $\partial^Cf(x)$, is defined via the generalized directional derivative by
\begin{align}
\nonumber \partial^Cf(x)=\{b\in\mathcal{H}: \langle b,d\rangle\leq 
 f^{\circ}(x, d)\;\mbox{ for all }\; d \in \mathcal{H}\}.
\end{align}
Note that, in the Hilbert space setting, we have the following relationship between the Clarke and Mordukhovich subdifferentials (see \cite[Theorem~3.57]{Morduk})
\begin{equation}\label{clco}
\partial^C f(x)={\rm clco}\,\partial^M f(x),
\end{equation}
where ``clco" stands for the weak closure of the convex hull of the set.

Define further the {\em distance function} $d:\mathcal{H} \to \mathbb{R}$ of $x\in \mathcal{H}$ to $\Omega\subset \mathcal{H}$ by
\begin{equation}
\nonumber d_{\Omega}(x):=\inf\{||x-y||\;:\;y \in \Omega\}.
\end{equation} 
Recall also that by a {\em lower semicontinuous vector function} $F:\mathcal{H}\to\mathbb{R}^m$, we mean that  component $f_i:\mathcal{H}\to\mathbb{R}$ of $F$ are lower semicontinuous.\vspace*{0.05in}

In our study, we will mostly use the limiting subdifferential $\partial^M$ whose graph is closed in rather general settings; see below. The regular subdifferential does not have the closed graph even for simple Lipschitz functions of one variable; see, e.g., $f(x):=-|x|$, where $\partial^F f(0)=\emptyset$. The usage of Clarke's subdifferential brings us to more restrictive results due the convexification in \eqref{clco}.\vspace*{0.05in} 

To begin with, we give a new definition of Pareto critical points for nonconvex problems, which is the limiting subdifferential extension of the classical notion for convex problems of multiobjective optimization. 

\begin{definition}\label{crit-pareto}
We say that $\bar{x}\in \Omega\subset \mathcal{H}$ is a {\em Pareto critical point} of the optimization problem (\ref{eq:mp}) if there exist $i\in \mathcal{I}$ and $\omega_i\in \partial^Mf_i(\bar{x})$ such that
\begin{align}
\nonumber \langle \omega_i, y-\bar{x}\rangle\geq 0\;\mbox{ for all }\;y\in \Omega.
\end{align}
\end{definition}

To proceed with the corresponding geometric constructions, the {\em regular/Fr\'echet normal cone} to a nonempty set $\Omega$ at $x\in\Omega$ is given by
\begin{align}
\nonumber N^F(x;\Omega):=\left\{ \omega\in \mathcal{H}: \; \limsup_{y\xrightarrow{\Omega} x}\dfrac{\langle \omega, y-x\rangle}{\|y-x\|}\leq 0\right\},
\end{align}
where the notation $y\xrightarrow{\Omega} x$ means $y\to x$ with $y\in \Omega$.
The {\em limiting/Mordukhovich normal cone} to $\Omega$ at $x\in\Omega$ is defined by 
\begin{align}
\nonumber N^M(x;\Omega)&:= \Limsup_{y\to x}N^F(y;\Omega)\\
\nonumber &=\{\omega\in\mathcal{H}:\, \exists \ x^k\xrightarrow{\Omega} x, \exists \ \omega_k\in N^F(x_k;\Omega); \ \omega_k \rightharpoonup \omega\}.
\end{align}
It is easy to see that $N^F(x;\Omega)\subset N^M(x;\Omega)$.

\begin{proposition}\label{lemma1}
Let $\Omega\subset\mathcal{H}$ be a nonempty closed set, and let $\delta_{\Omega}$ be the indicator function on $\Omega$. Then we have 
\begin{align}
\nonumber \partial^M\delta_{\Omega}(x)=N^M(x;\Omega).
\end{align}
\end{proposition}
\begin{proof}
See \cite[Proposition~1.79]{Morduk}.
\end{proof}

The following theorem is the fundamental calculus rule.

\begin{theorem}\label{lemma3}
Let $f,g:\mathcal{H}\to\mathbb{R}$ be functions such that $f$ is locally Lipschitz around some point $x$ and $g$ is lower semicontinuous around this point.
Then 
\begin{align}
\nonumber \partial^M (f+g)(x)\subset  \partial^M f(x)+\partial^Mg(x).
\end{align}
\end{theorem}
\begin{proof}
See \cite[Theorem 3.36]{Morduk}.
\end{proof}
The next result ensures that the graph of Mordukhovich's subdifferential is a closed set in the norm-weak topology of ${\mathcal H}\times{\mathcal H}$ when $f$ is locally Lipschitz around the point in question. This fact is essential in the convergence analysis of the designed algorithm to find Pareto critical points.

\begin{theorem}\label{weakgraph}
Let $\{x^k\},\{\omega^k\}\subset \mathcal{H}$ be sequences such that $\omega^k\in\partial^Mf(x^k)$ for all $k \in \mathbb{N},\, x^k\rightharpoonup x$, and $\omega^k\rightharpoonup\omega$ as $k\to\infty$. If $f:\mathcal{H}\to\mathbb{R}$ is locally Lipschitz around $x$, then we have the inclusion $\omega\in\partial^Mf(x)$.
\end{theorem}
\begin{proof} It is well known in geometric theory of Banach spaces that any Hilbert space is Asplund and weakly compactly generated (WCG); see, e.g., \cite[Chapter~2]{Morduk} with the further references and commentaries therein. Thus we are in the setting of  \cite[Theorem 3.60]{Morduk}. This theorem tells us that we have the claimed subdifferential graph closedness if the graph of the subdifferential mapping $\partial^Mf$ in $\mathcal{H}\times \mathcal{H}$ enjoys the {\em sequential normal compactness} (SNC) property from \cite[Definition~1.20]{Morduk}. This property is fairly general and holds, in particular, when the set in question  is {\em epi-Lipschitzian} in the sense of Rockafellar \cite{Roc79}.  As shown in \cite{Roc79}, the latter is surely the case for the graphs of locally Lipschitz functions. Thus we are done with the proof of the closed-graph theorem.
\end{proof}

\begin{remark}\label{closed-graph} The aforementioned result of \cite[Theorem~3.60]{Morduk} allows us to get the closed-graph property of $\partial^Mf$ for significantly larger classes of functions than just the locally Lipschitz ones. Observe that the SNC property holds {\em automatically} in finite-dimensional spaces, which allows us to cover in the corresponding counterpart of Theorem~\ref{weakgraph} a very broad class of extended-real-valued {\em subdifferentially continuous} \cite{rw} functions in finite dimensions. In infinite-dimensional spaces, even appealing to the epi-Lipschitzian property of the limiting subdifferential graph, we cover the class of (locally) {\em directionally Lipschitzian} functions, which are much broader than the standard locally Lipschitz class; see \cite{Roc79}. Moreover, we know from \cite[Theorem~1.26]{Morduk} that, in any Banach space, a set is SNC around a given point if it is {\em compactly epi-Lipschitzian} (CEL) around this point in the sense of Borwein-Str\'ojwas \cite{borwein}. This creates the opportunity of extending the main algorithmic and related results of this paper to truly non-Lipschitzian functions in finite and infinite dimensions.
\end{remark}

To proceed with the study of weak Pareto points, we present the following lemma useful in deriving necessary conditions for weak Pareto optimality.

\begin{lemma}\label{lema3.1}
Let $\mathcal{D}:=\{x\ \in \mathcal{C}: g_j(x)\leq 0, j\in \mathcal{J}\}$, where $g_j:\mathcal{H}\to\mathbb{R}$ and $\mathcal{J}=\{1,\cdots,p\}$. If $x^*$ is a weak Pareto point of $F$ in $\mathcal{D}$, then for all $\varepsilon>0$ and $x\in \mathcal{C}$ we have the positivity condition
\begin{align}
\nonumber F_{\varepsilon}(x):=\max\{f_i(x)-f_i(x^*)+\varepsilon, g_j(x): i\in\mathcal{I},\, j\in \mathcal{J}\}>0.
\end{align}
\end{lemma}
\begin{proof}
Suppose the contrary; i.e., there exist $\varepsilon_0$ and $x_1 \in \mathcal{C}$ such that
\begin{align}
\nonumber F_{\varepsilon_0}(x_1)\leq 0.
\end{align}
Therefore, it follows that
\begin{align}
\nonumber f_i(x_1)<f_i(x_1)+\varepsilon_0\leq  f_i(x^*), \quad  i\in\mathcal{I},\\
\nonumber g_j(x_1)\leq 0, \quad  j\in\mathcal{J},
\end{align}
and so $x_1 \in \mathcal{D}$. This contradicts the weak Pareto optimality of  $x^*$ and hence verifies the claim of the lemma. 
\end{proof}

The next result establishes new {\em optimality conditions} of the Fritz-John type, provided by Mordukhovich's subdifferential, which are necessary for weak Pareto optimal solutions used in the PPA while being of their own interest.

\begin{theorem}\label{teoCO}
Let $\mathcal{C}\subset \mathcal{H}$ be a nonempty, closed, and convex set. Assume that the set $\mathcal{D}$ is represented in the form
\begin{equation*}
\mathcal{D}=\{x\in \mathcal{C}: g_j(x)\leq 0, j\in J\},
\end{equation*}
where the functions $f_i,g_j:\mathcal{H}\to\mathbb{R}$ are locally Lipschitz around the reference point for all $i\in \mathcal{I}$ and $j\in \mathcal{J}$. If $x^*\in\mathcal{D}$ is a weak Pareto solution of \eqref{eq:mp}, then there exist real numbers $\tau>0$ and $\alpha_i\geq 0$, $\beta_j\geq 0$ with $i \in \mathcal{I}$ and $j \in \mathcal{J}$ such that we have the conditions
\begin{equation}
\nonumber 0\in \sum_{i\in \mathcal{I}}\alpha_i \partial^M f_i(x^*)+\sum_{j\in \mathcal{J}}\beta_j \partial^M g_j(x^*)+\tau\partial^M d_{\mathcal{C}}(x^{*}),
\end{equation}
\begin{equation}
\nonumber\mbox{with}\quad\sum_{i\in \mathcal{I}}\alpha_i+\sum_{j\in \mathcal{J}}\beta_j=1\quad\mbox{and}\quad \beta_j g_j(x^*)=0, \quad j\in \mathcal{J}.
\end{equation}
\end{theorem} 
\begin{proof}
Let $x^*$ be a weak Pareto point of $F$ in $\mathcal{D}$, and let $\{\varepsilon_l\}\subset \mathbb{R}_{++}$ be a sequence converging to 0 as $l\to\infty$. It follows from Lemma~\ref{lema3.1} that
\begin{align}
\nonumber F_{\varepsilon_l}(x)>0\;
\mbox{ for all }\; l\in \mathbb{N},\;\;x\in \mathcal{C}.
\end{align}
Therefore, the function $F_{\varepsilon_l}:\mathcal{C}\to \mathbb{R}_{+}$ is bounded 
from below; in particular, when $0<\varepsilon_l\leq 1$ in our consideration. The imposed local Lipschitz continuity of $f_i,g_j$  around $x^*$ for all $i\in \mathcal{I}$ and $j\in \mathcal{J}$ ensures that  $F_{\varepsilon_l}$ is also locally Lipschitz around this point. Furthermore, as $x^* \in \mathcal{D}$ and $F_{\varepsilon_l}(x)>0$ for all $x\in\mathcal{C}$, we get that $\displaystyle\inf_{x\in\mathcal{C}} F_{\varepsilon_l}(x)\geq 0$, which tells us in turn that
\begin{align}
\nonumber \varepsilon_l=F_{\varepsilon_l}(x^*)\leq \inf_{x\in\mathcal{C}} F_{\varepsilon_l}(x)+\varepsilon_l.
\end{align}
The latter means that $F_{\varepsilon_l}$ satisfies all the hypotheses of \cite[Lemma~2.2 ]{minami}, which is an appropriate version of the fundamental {\em Ekeland variational principle}; see, e.g., \cite[Theorem~2.26]{Morduk}. In this way, we find $z^l\in\mathcal{C}$ such that
\begin{align}
\|x^*-z^l\|\leq \sqrt{\varepsilon_l}\;\mbox{ and}\label{eq2.8}
\end{align}
\begin{align}
\nonumber F_{\varepsilon_l}(x)+\sqrt{\varepsilon_l}\|x-z^l\|\geq F_{\varepsilon_l}(z^l)\;\mbox{  for all }\;x \in \mathcal{C},
\end{align}
i.e., $z^l$ is a global solution of $F_{\varepsilon_l}(\cdot)+\sqrt{\varepsilon_l}\|\cdot-z^l\|$ in $\mathcal{C}$. Denoting by $L$ a Lipschitz constant of $F_{\varepsilon_l}(\cdot)$, we get that $L+1$ is a Lipschitz constant of $F_{\varepsilon_l}(\cdot)+\sqrt{\varepsilon_l}\|\cdot-z^l\|$. Thus, by \cite[Lemma~2.1]{minami}, $z^l$ is a local optimal solution of the unconstrained problem of minimizing the function
\begin{align}
\nonumber G_{\varepsilon_l}(x)= F_{\varepsilon_l}(x)+\sqrt{\varepsilon_l}\|x-z^l\|+(L+1)d_{\mathcal{C}}(x).
\end{align}
Applying to the latter problem the generalized Fermat rule in terms of the limiting subdifferential (see \cite[Proposition~1.114]{Morduk}) and then the subdifferential sum rule from Lemma~\ref{lemma3}, we arrive at 
\begin{align}
\nonumber 0 \in \partial^M G_{\varepsilon_l}(z^l)\subset \partial^MF_{\varepsilon_l}(z^l)+\sqrt{\varepsilon_l}\partial^M(\|\cdot-z^l\|)(z^l)+(L+1)\partial^Md_{\mathcal{C}}(z^l).
\end{align}
Therefore,  for each $l$ there exists $\omega^l\in\partial^MF_{\varepsilon_l}(z^l)$ satisfying
\begin{align}
0 \in \omega^l+\sqrt{\varepsilon_l}\partial^M(\|\cdot-z^l\|)(z^l)+(L+1)\partial^Md_{\mathcal{C}}(z^l).\label{eq2.9}
\end{align}
Using \cite[Theorem~3.46]{Morduk} and then \cite[Theorem~3.52]{Morduk} gives us nonnegative numbers $u_i^l$ and $v_j^l$ such that
\begin{align}
&\sum_{i\in \mathcal{I}}\alpha^l_i+\sum_{j\in \mathcal{J}}\beta^l_j=1,\label{eq2.10}\\
\omega^l \in \sum_{\{i| \alpha^l_i\neq 0\}}&\alpha_i^l\partial^Mf_i(z^l)+\sum_{\{j| \beta^l_j\neq 0\}}\beta_j^l\partial^Mg_j(z^l).\label{eq2.11}
\end{align}
Combining (\ref{eq2.9}) and (\ref{eq2.11}) ensures that
\begin{equation*}
\begin{array}{ll}
0\in \displaystyle\sum_{\{i| \alpha^l_i\neq 0\}}\alpha_i^l\partial^Mf_i(z^l)+\sum_{\{j| \beta^l_j\neq 0\}}\beta_j^l\partial^Mg_j(z^l)\\
\qquad\displaystyle+\sqrt{\varepsilon_l}\partial^M(\|\cdot-z^l\|)+(L+1)\partial^Md_{\mathcal{C}}(z^l).
\end{array}
\end{equation*}
In this way, we find sequences $\{u_i^l\},\, \{v_j^l\},\,  \{\tilde{u}^l\}$, and $ \{\tilde{v}^l\}$ with
$$
u_i^l\in  \partial^Mf_i(z^l), \;v_j^l\in \partial^Mg_j(z^l),\;\tilde{u}^l\in \sqrt{\varepsilon_l}\partial^M(\|\cdot-z^l\|),
$$
and $ \tilde{v}^l\in (L+1)\partial^Md_{\mathcal{C}}(z^l)$ such that
\begin{align}
0=\sum_{\{i| \alpha^l_i\neq 0\}}&\alpha_i^lu_i^l+\sum_{\{j| \beta^l_j\neq 0\}}\beta_j^lv_j^l+ \tilde{u}^l+ \tilde{v}^l.\label{eq2.12}
\end{align}
If follows  from (\ref{eq2.10}) and (\ref{eq2.8}) that the sequences $\{\alpha_i^l\}_l$, $\{\beta_j^l\}$, and $\{z^l\}$ are bounded. Furthermore, the subgradient boundedness of $\partial^M$ by the local Lipschitz continuity of $f_i$, $g_j$, and $\|\cdot-z^l\|$ (see \cite[Corollary 1.81]{Morduk}) ensures the boundedness of the sequences $\{u_i^l\}$, $\{v_j^l\}$, and $\{\tilde{u}^l\}$. The boundedness of $\{\tilde{v}^l\}$ follows from (\ref{eq2.12}). Passing to subsequences if necessary, we get that the sequences $\{\alpha_i^l\}$, $\{\beta_j^l\}$, $\{u_i^l\}$, $\{v_j^l\}$, $\{\tilde{u}^l\}$, and $\{\tilde{v}^l\}$ weakly converge, respectively, to $\alpha_i$, $\beta_j$, $u_i$, $v_j$, $ \tilde{u}$, and $ \tilde{v}$. Hence $\alpha_i,\;\beta_j\geq 0$ and $\sum_{i\in \mathcal{I}}\alpha_i+\sum_{j\in \mathcal{J}}\beta_j=1$. By (\ref{eq2.8}), we have that $z^l\to x^*$. Then (\ref{eq2.12}) as $l\to\infty$ gives us
\begin{align}
\nonumber 0=\sum_{i\in\mathcal{I}}\alpha_iu_i+\sum_{j\in\mathcal{J}}\beta_jv_j+\tilde{u}+\tilde{v}.
\end{align}
The closed-graph property of Theorem~\ref{weakgraph} yields
\begin{align}
\nonumber u_i \in \partial^Mf_i(x^*), \quad v_j\in \partial^Mg_j(x^*),\quad\mbox{and}\quad \tilde{v} \in \partial^Md_{\mathcal{C}}(x^*),
\end{align}
and therefore, we arrive at the claimed inclusion
\begin{align}
\nonumber 0\in\sum_{i\in \mathcal{I}}\alpha_i \partial^M f_i(x^*)+\sum_{j\in \mathcal{J}}\beta_j \partial^M g_j(x^*)+\tau\partial^M d_{\mathcal{C}}(x^{*})\;\mbox{ with some }\;\tau>0.
\end{align}
\end{proof}

To proceed further, recall the following result from \cite[Theorem~1.97]{Morduk}.

\begin{lemma}\label{lemma4}
Let $\Omega \subset \mathcal{H}$ be nonempty and closed. Then
$$
N^M(x;\Omega)=\displaystyle\cup_{\lambda>0}\lambda\partial^M d_{\Omega}(x)\;\mbox{ for any }\;x \in \Omega.
$$
\end{lemma}

Here is the subdifferential estimate for the distance function, which plays an important role in the subsequent convergence analysis.

\begin{theorem}\label{teo2}
Let $\Omega$ be a nonempty, closed, and convex subset of $\mathcal{H}$. Then for any $x \in \Omega$, we have the upper estimate
\begin{align}
\nonumber \partial^M d_{\Omega}(x)\subset B[0,1]\cap N^M(x;\Omega),
\end{align}
where $B[0,1]$ denotes the closed unit ball in $\mathcal{H}$.
\end{theorem}
\begin{proof}
It is known that the distance function $d_{\Omega}(x)$ is globally Lipschitz with a constant $L=1$. It follows from \cite[Corollary~1.81]{Morduk} that $\partial^M d_{\Omega}(x)\subset B[0,1]$. Employing now Lemma~\ref{lemma4} tells us that $N^M(x;\Omega)=\displaystyle\cup_{\lambda>0}\lambda\partial^M d_{\Omega}(x)$ , and thus we arrive at $\partial^M d_{\Omega}(x)\subset N^M(x;\Omega)$. 
\end{proof}

\section{ The proximal point algorithm}\label{algorithm}
In this section, we prove some facts related to our new approach for convergence of the Proximal Point Algorithm (PPA) for vector-valued functions. 
 
\subsection{The algorithm}

In what follows, $\Omega\subset \mathcal{H}$ is a nonempty, closed, and convex set while $F=(f_1,\ldots,f_m): \mathcal{H} \to\mathbb{R}^m$ is a vector function such that its each component $f_i:H\to\mathbb{R}$, $i\in\mathcal{I}$, is locally Lipschitz. From Remark~\ref{exp}, we can assume without loss of generality that $F\succ 0$. 

To formulate the PPA for finding a Pareto critical point of $F$ in $\Omega$, let $\{\lambda_{k}\}$ be a sequence of positive numbers, and let $\{\varepsilon^k\}\subset\mathbb{R}_{++}^{m}$ be a sequence with $||\varepsilon^k||=1$ for all $k\in\mathbb{N}$. The method generates an iterative sequence $\{x^k\}\subset \Omega$ in the following way:\\

\noindent{\bf Algorithm (refined vectorial PPA)}\\[0.5ex]
{\em Initialization}: Choose $x^0\in \Omega$.\\[0.5ex]
{\em Stopping rule}: If for $x^ k \in \Omega$ we have that $x^k$ is a Pareto critical point, then set $x^{k+p}:=x^k$ for all $p \in \mathbb{N}$.\\[0.5ex]
{\em Iterative step}: As the next iterate, take $x^{k+1}\in\Omega$ such that
\begin{equation}
\nonumber x^{k+1}\in \mbox{argmin}_{w}\left\{ F(x)+\frac{\lambda_k}{2}||x-x^k||^2\varepsilon^k\;:\;x\in\Omega_k\right\}, 
\end{equation}
where $\Omega_k:=\{x \in \Omega\;:\;F(x)\preceq F(x^k)\}$.\\[0.1ex]

We would like to mention that this method aims at finding separate solutions while not the entire solution set. It has been noticed by Fukuda and Gra\~na Drummond~\cite{sobrapo} and  by Fliege et al.~\cite{Benar2009} that we can expect to somehow approximate the solution set by just performing this method for different initial points. In the well-known weighting method, this kind of idea also appears in Burachik et al.~\cite{Burachik}. More precisely, the method can be performed for different weights in order to find either the solution set, or a reasonable approximation of this set. However, in some cases, arbitrary choices of the weighting vectors may lead the weighting method to unbounded problems. The Pareto front, i.e., the objective values of these solutions, is in general an infinite set. Hence, in practice, only an approximation of the Pareto front is obtained.
 
\begin{proposition}
The Algorithm  is well-defined.
\end{proposition}
\begin{proof}
The starting point $x^0\in\Omega$ is chosen in the initialization step. Assuming in iteration $ k $ that $ x^k$ is given and denote 
$$
F_k(x):=F(x)+\frac{\lambda_k}{2}\|x-x^k\|^2 \varepsilon^k.
$$
Note that we have $x^k\in\Omega_k$, which implies that $F_k(\Omega_k)$ is nonempty. It is straightforward to check that $F_k(x)\succ 0$ and $F_k(\Omega_k)$ is closed. Since the cone $\mathbb{R}^m_+$ is Daniell, it follows from \cite[Lemma~3.5]{DinhLuc1989} that $F_k(\Omega_k)$ is $\mathbb{R}^m_+$-complete. From \cite[Theorem~3.3]{DinhLuc1989}, the next iteration exists and satisfies the inclusion
$$
x^{k+1}\in \mbox{arg\,min}_w \{F_k(x)\,:\,x\in \Omega_k\}.
$$
which completes the proof of the proposition.
\end{proof}

From now on, the sequences $\{x^{k}\}$, $\{\lambda_{k}\}$, and $\{\varepsilon^k\}$ are taken from the above Algorithm. The following hold.

\begin{proposition}\label{meanprop}
For all $k\in \mathbb{N}$, there exist $\alpha^k,\beta^k\in \mathbb{R}^m_+$, $\omega^k,u_i^k\in \mathcal{H}$, and $\tau_{k}\in \mathbb{R}_{++}$ satisfying the equalities
\begin{equation}\label{condotim}
\sum_{i\in\mathcal{I}}(\alpha^{k}_i+\beta^{k}_i)u_i^k+\lambda_{k-1}\sum_{i\in\mathcal{I}}  \varepsilon^{k-1}_i\alpha^{k}_i (x^{k}-x^{k-1})+\tau_{k}\omega^{k}=0
\end{equation}
together with the conditions
\begin{equation}
\nonumber\quad \omega^{k}\in B[0,1]\cap N^M(x^{k};\Omega), \quad u^k_i\in\partial^Mf_i(x^k),\quad \mbox{and}\quad ||\alpha^k+\beta^k||_{1}=1.
\end{equation}
\end{proposition}
\begin{proof}
It follows from the constructions of the Algorithm that  $x^{k}$ is a weak Pareto solution of the multiobjective problem
\begin{center}
$\min_{w}\{F_{k-1}(x)\; : \; x \in \Omega_{k-1}\}$,
\end{center}
where $F_{k-1}(x):=F(x)+\frac{\lambda_{k-1}}{2}||x-x^{k-1}||^2\varepsilon^{k-1}$ for each $k\in\mathbb{N}$. 

Defining $G_{k-1}(x):=F(x)-F(x^{k-1})$, it is easy to deduce from the local Lipschitz continuity of $F$ that all the component functions 
\begin{equation}\label{gtilde}
(g_{k-1})_i(\cdot)=f_{i}(\cdot)-f_{i}(x^{k-1}), \qquad i\in \mathcal{I},
\end{equation}
with $(f_{k-1})_i$ given by 
\begin{equation}\label{ftilde}
(f_{k-1})_i(\cdot)=f_i(\cdot)+\frac{\lambda_{k-1}}{2}||\cdot-x^{k-1}||^2 \varepsilon^{k-1}_{i},\qquad i\in \mathcal{I},
\end{equation}
are locally Lipschitzian. Fix $k\in\mathbb{N}$ and apply Lemma~\ref{lemma3} and Theorem~\ref{teoCO} with $g_j$ and $f_j$ defined in $\eqref{gtilde}$ and $\eqref{ftilde}$, respectively.  Taking into account that $\partial^M d_{\Omega}(x^{k})\subset B[0,1]\cap N^M(x^{k};\Omega)$, $k\in\mathbb{N}$, by Theorem~\ref{teo2}  ensures that
\begin{align}
\nonumber 0=\sum_{i\in\mathcal{I}}(\alpha_i^k+\beta_i^k)u_i^k+\lambda_{k-1}\sum_{i\in\mathcal{I}}\alpha_i^k\varepsilon_i^{k-1}(x^k-x^{k-1})+\tau_k\omega^k,
\end{align}
where the elements involved satisfy the conditions
$$
u_i^k\in\partial^Mf_i(x^k), \quad \omega^k\in \partial^Md_{\Omega}(x^k)\subset B[0,1]\cap N^M(x^{k};\Omega),\quad\displaystyle\sum_{i\in\mathcal{I}}(\alpha_i^k+\beta_i^k)=1
$$ 
for all $k$. This completes the proof of the proposition.
\end{proof} 

As in \cite{Bonnel2005}, the stopping rule in the Algorithm can be changed by the following rule, which is easier to check: after computing $x^{k+1}$, the Algorithm stops if $x^{k+1}=x^k$, i.e., we set $x^{k+p}:=x^k$ for all $p\ge 1$. 

\begin{corollary}\label{coro:002}
If $x^{k+1}=x^k$, then $x^k$ is a Pareto critical point of $F$.
\end{corollary}
\begin{proof}
Assume that $x^{k+1}=x^k$. Then Proposition~\ref{meanprop} tells us that there exist $\alpha^{k+1},\;\beta^{k+1} \in\mathbb{R}^m_{+},\;u_i^{k+1}, \omega^{k+1} \in \mathcal{H}$, and $\tau_{k+1}\in\mathbb{R}_{++}$ satisfying
\begin{align}
\nonumber  \sum_{i\in\mathcal{I}}(\alpha_i^{k+1}+\beta_i^{k+1})u_i^{k+1}+\tau_{k+1}\omega^{k+1}=0,
\end{align}
where $-\sum_{i\in\mathcal{I}}(\alpha_i^{k+1}+\beta_i^{k+1})u_i^{k+1}\in N^M(x^{k+1};\Omega)$. If $x^k$ is not a Pareto critical point  of $F$ in $\Omega$, then there exists $y\in \Omega$ such that $\langle u_i^{k+1}, y-x^{k+1}\rangle<0$. Therefore, with $\displaystyle\sum_{i\in\mathcal{I}}(\alpha_i^{k+1}+\beta_i^{k+1})=1$ we get
\begin{align}
\nonumber \left\langle \sum_{i\in\mathcal{I}}(\alpha_i^{k+1}+\beta_i^{k+1}) u_i^{k+1}, y-x^{k+1}\right\rangle<0,
\end{align}
which contradicts $-\sum_{i\in\mathcal{I}}(\alpha_i^{k+1}+\beta_i^{k+1})u_i^{k+1}\in N^M(x^{k+1};\Omega)$ and thus verifies that $x^k$ is a Pareto critical point of $F$ in $\Omega$.
\end{proof}

\begin{corollary}
If there exists  a number $k_0 \in \mathbb{N}$ such that $\alpha^{k_0}=0$, then $x^{k_0}$ is a Pareto critical point of $F$. 
\end{corollary}
\begin{proof}
Assume that there exists $k_0\in \mathbb{N}$ such that $\alpha^{k_0}=0$. Then Proposition~\ref{meanprop} provides the equality
\begin{align}
\nonumber \sum_{i\in\mathcal{I}}\beta_i^{k_0}u_i^{k_0}+\tau^{k_0}\omega^{k_0}=0
\end{align}
with $\omega^{k_0}\in B[0,1]\cap N^M(x^{k_0};\Omega)$, $u_i^{k_0}\in \partial^Mf_i(x^{k_0})$, and $ \displaystyle\sum_{i\in\mathcal{I}}\beta_i^{k_0}=1$. Therefore, we get $\beta^{k_0}\in \mathbb{R}^m_+$. Supposing now that $x^{k_0}$ is not a Pareto critical point of $F$ in $\Omega$ ensures the existence of $y\in \Omega$ such that for all $i\in \mathcal{I}$ and $u_i^{k_0}\in\partial^Mf_i(x^{k_0})$ we have the strict inequality
\begin{align}
\nonumber \langle u_i^{k_0}, y-x^{k_0}\rangle<0.
\end{align}
Finally, it follows from $\beta_i^{k_0}\in \mathbb{R}^m_+$ and $\displaystyle\sum_{i\in\mathcal{I}}\beta_i^{k_0}=1$ that
\begin{align}
\nonumber \left\langle \sum_{i\in\mathcal{I}}\beta_i^{k_0}u_i^{k_0}, y-x^{k_0}\right\rangle<0,
\end{align}
which contradicts the fact that $-\displaystyle\sum_{i\in\mathcal{I}}\beta_i^{k_0}u_i^{k_0}\in N^M(x^{k_0};\Omega)$.
\end{proof}

\begin{remark}
Note that if the Algorithm terminates after a finite number of iterations, it terminates at a Pareto critical point. This allows us us to suppose that $\{x^{k}\}$ is an infinite sequence, and hence  $x^{k+1}\neq x^{k}$ for all $k\in\mathbb{N}$ due to the result of Corollary~\ref{coro:002}.
\end{remark}

\section{Convergence analysis}\label{sec:4}
In this section, we employ the following  $\mathbb{R}^m_+$-completeness assumption imposed on the vector function $F$ and the initial point $x^0$:

{\bf($\mathbb{R}^m_+$-Completeness assumption)} For all sequences $\{a^k\}\subset \mathcal{H}$ satisfying  $a^0=x^0$ and $F(a^{k+1})\preceq F(a^k)$ for all $k \in \mathbb{N}$, there exists $a \in \mathcal{H}$ such that 
$$
F(a)\preceq F(a^k)\;\mbox{ for all }\; k \in \mathbb{N}.
$$

In the classical case of unconstrained convex optimization, this condition is equivalent to the existence of solutions to the optimization problem in question. In general, the $\mathbb{R}^m_+$-completeness assumption is recognized to ensure the existence of efficient points for vector optimization problems; see, e.g., \cite[page 46]{DinhLuc1989}. An interesting discussion on the existence of efficient points for vector optimization problems can also be found in \cite[Chapter 2]{DinhLuc1989}. In what follows, we employ this conditions to establishing the convergence of our Algorithm to find Pareto critical points in problems of multiobjective optimization.

\subsection{Locally Lipschitz case}

 First we establish our convergence result in the locally Lipschitz case.

\begin{theorem}\label{maintheorem}
Suppose that there exist scalars $a,b,c\in\mathbb{R}_{++}$ such that $0<a\leq\lambda_k\leq b$ and $0<c\leq\varepsilon_j^k$, for all $k\in\mathbb{N}$ and $j=1,\ldots,m$. Then every weak cluster point of $\{x^k\}$ is a Pareto critical point of $F$.
\end{theorem}
\begin{proof}
It follows from the definition of the Algorithm that $x^k$, for each fixed $k\in \mathbb{N}$, is an optimal solution of the problem
\begin{align}
\nonumber \mbox{min}_{w}\left\{ F(x)+\frac{\lambda_{k-1}}{2}\|x-x^{k-1}\|^2\varepsilon^{k-1}\;:\;x\in\Omega_{k-1}\right\}.
\end{align}
This readily yields the condition
\begin{align}
\nonumber \max_{1\leq j\leq m}\left\{ f_j(x^{k-1})-f_j(x^{k})-\frac{\lambda_{k-1}}{2}\|x^k-x^{k-1}\|^2\varepsilon^{k-1}\right\}\geq 0.
\end{align}
Take $j_0(k) = j_0 \in \{1, . . .,m\}$ as an index where the maximum in the last inequality is attained. Then we deduce from the lower boundedness assumption on the sequences $\{\lambda_k\}$ and $\{\varepsilon^k\}$ that
\begin{align}
\nonumber\dfrac{ac}{2}\|x^k-x^{k-1}\|&\leq  f_{j_0}(x^{k-1})-f_{j_0}(x^{k})\\
\nonumber &=\sqrt{( f_{j_0}(x^{k-1})-f_{j_0}(x^{k}))^2}\\
\nonumber &\leq \sqrt{\sum_{j\in\mathcal{I}}( f_{j}(x^{k-1})-f_{j}(x^{k}))^2}\\
&=\|F(x^{k-1})-F(x^{k})\|, \quad k\in \mathbb{N}.\label{ass}
\end{align}
Since $\{F(x^k)\}$ is nonincreasing and $F\succ 0$, we deduce that the right-hand side of (\ref{ass}) converges to $0$ as $k\to\infty$. This gives us 
\begin{align}\label{eq.18}
\|x^k-x^{k-1}\|\to 0 \;\; as \;\;k\to\infty.
\end{align}
Now let $\hat{x}$ be a weak cluster point of $\{x^k\}$, and let $\{x^{k_l}\}$ be a subsequence of $\{x^k\}$ converging to $\hat{x}$. Applying Proposition~2 to the sequence $\{x^{k_l}\}$, we find
\begin{equation*}
\left\{
\begin{array}{l}
\{u_i^{k_l}\}\subset \partial^Mf_i(x^{k_l}),\\
\{\alpha^{k_l}\}, \{\beta^{k_l}\}\subset \mathbb{R}^m_+,\\
\{\omega^{k_l}\}\subset\mathcal{H},\quad \mbox{and} \quad \{\tau_{k_l}\}\subset \mathbb{R}_{++}
\end{array} \right.
\end{equation*}
satisfying the equalities
\begin{align}\label{meanteoinq2}
\sum_{i\in\mathcal{I}}(\alpha_i^{k_l}u_i^{k_l}+\beta_i^{k_l})+\lambda_{k_l-1}\sum_{i\in\mathcal{I}}  \varepsilon^{k_l-1}_i\alpha^{k_l}_i (x^{k_l}-x^{k_l-1})+\tau_{k_l}\omega^{k_l}=0,
\end{align}
for all $k\in\mathbb{N}$, where 
\begin{align}\label{eq.20}
\omega^{k_l}\in B[0,1]\cap N^M(x^{k_l};\Omega)\quad \mbox{and}\quad \displaystyle \sum_{i\in\mathcal{I}}(\alpha^{k_l}+\beta^{k_l})=1.
\end{align} 
Since $\{x^{k_l}\}$ weakly converges to $\hat{x}$, it follows that $\{x^{k_l}\}$ is bounded.
Using the local Lipschitz continuity of $F$ and taking into account that $\partial^Mf_j$ in \eqref{eq.20} is bounded on compact sets by \cite[Corollary~1.81]{Morduk}, we get that
the sequences $ \{\alpha^{k_l}\},\, \{\beta^{k_l}\},\, \{\omega^{k_l}\}$ are bounded. Thus $\{\tau_{k_l}\}$ is bounded as well. Without loss of generality, assume that $\{u_i^{k_l}\},\, \{\alpha^{k_l}\}, \,\{\beta^{k_l}\},\, \{\omega^{k_l}\}$ weakly converge to $u_i,\, \alpha, \, \beta,\, \omega $, respectively, and that $\{\tau_{k_l}\}$ converge to $\tau$. Furthermore, it follows from the boundedness of $\lambda_{k_l-1}\sum_{i\in\mathcal{I}}  \varepsilon^{k_l-1}_i\alpha^{k_l}_i$ together with (\ref{eq.18}) that
\begin{align*}
\sum_{i\in\mathcal{I}}(\alpha_i+\beta_i)u_i+\tau\omega=0.
\end{align*}
Employing Lemma~\ref{lemma1} and Theorem~\ref{weakgraph} gives us $u_i\in\partial^Mf_i(\hat{x})$ and $\omega\in N^M(\hat{x};\Omega)$. Thus we arrive at the inclusion
\begin{align}
-\sum_{i\in\mathcal{I}}(\alpha_i+\beta_i)u_i\in N^M(\hat{x};\Omega).\label{eq.36}
\end{align}
Arguing finally by contradiction, suppose that $\hat{x}$ is not a Pareto critical point of $F$. Then there exists $y\in\Omega$ such that
\begin{align*}
\langle u_i,y-\hat{x}\rangle<0\;\mbox{ for all }\;i\in\mathcal{I}.
\end{align*}
This implies therefore the strict inequality
$$
\left\langle \displaystyle\sum_{i\in\mathcal{I}}(\alpha_i+\beta_i)u_i, y-\hat{x}\right\rangle<0.
$$
Remembering that $\alpha,\beta\in\mathbb{R}^m_+$ and $\displaystyle\sum_{i\in\mathcal{I}}(\alpha_i+\beta_i)=1$, we get a contradiction with (\ref{eq.36}) and hence complete the proof of the theorem.
\end{proof}

\subsection{Pseudoconvexity case}

This subsection presents a constructive sufficient condition ensuring that the {\em entire sequence} of the iterates generated by the Algorithm weakly converges to a Pareto critical point of $F$.

Using the limiting subdifferential $\partial^M$, we say that a {\em scalar function} $f$ defined on a Hilbert space ${\mathcal H}$ is {\em pseudoconvex} if for any $x,y\in \mathcal{H}$ we have 
\begin{align*}
\exists\, \omega\in \partial^Mf(x): \langle \omega, y-x\rangle\geq 0\Rightarrow f(x)\leq f(y).
\end{align*}
Similarly to the differentiable case, it is not difficult to observe that any pseudoconvex function $f$ satisfies the following fundamental properties:
\begin{itemize}
\item[(a)]every local minimizer of $f$ is a global minimizer;
\item[(b)]$0\in \partial^M f(x) \Rightarrow f$ has a global minimum at $x$.
\end{itemize}

Further, we say that the {\em vector function} $F=(f_1,\ldots,f_m): \mathcal{H} \to\mathbb{R}^m$ is {\em pseudoconvex} if each component function $f_i:H\to\mathbb{R}$, $i\in\mathcal{I}$, is a pseudoconvex.

\begin{proposition}
Let $F : \mathcal{H} \to \mathbb{R}^m$ be a pseudoconvex function. If $x \in \Omega$ is a Pareto critical point of $F$ in $\Omega$, then it is a weak Pareto solution of $F$ in $\Omega$.
\end{proposition}
\begin{proof}
Assume that $x$ is a Pareto critical point of $F$ in $\Omega$. Arguing by contradiction, suppose that $x$ is not a weak Pareto solution of $F$ in $\Omega$. Then there exists a point $\tilde{x}\in\Omega$  such that
\begin{align*}
f_i(\tilde{x})<f_i(x)\;\mbox{ whenever }\;i\in\mathcal{I}.
\end{align*}
Since $f_i$ is pseudoconvex, it follows that for all $\omega_i\in\partial^Mf_i(x)$ and $i\in \mathcal{I}$ we have the strict inequality
\begin{align}
\nonumber \langle \omega_i, \tilde{x}-x\rangle<0.
\end{align}
However, this contradicts the fact that $x$ is a Pareto critical point of $F$ in $\Omega$.
\end{proof}

Before deriving the main result of this subsection, recall that a sequence $\{y^k\}$ is said to be {\em Fej\'er convergent} (or {\em Fej\'er monotone}) to a nonempty set $U\subset \mathcal{H}$ if for all $k \in \mathbb{N}$ we have
$$
||y^{k+1}-y||\leq ||y^k-y||\;\mbox{ whenever }\;y \in U.
$$

The following result is well known, and its proof is elementary.

\begin{proposition}\label{fejer}
Let $U\subset \mathcal{H}$ be a nonempty set, and let $\{y^k\}$ be a Fej\'er convergent sequence to $U$. Then $\{y^k\}$ is bounded. Moreover, if a weak cluster point $y$ of $\{y^k\}$ belongs to $U$, then $\{y^k\}$ converges to $y$ as $k\to\infty$.
\end{proposition}

The next theorem shows that in the pseudoconvex case, we have the weak convergence of the entire sequence of iterations generated by the Algorithm to a Pareto critical point.

\begin{theorem}
Under the imposed assumptions, the sequence $\{x^k\}$ generated by the Algorithm weakly converges to a  weak Pareto solution of F in $\Omega$ as $k\to\infty$.
\end{theorem}
\begin{proof}
We split the proof of the theorem into the following five steps.\\
{\it Step}~1 ({\em Fej\'er convergence}) Define the set $E\subset \Omega$ by
$$
E:=\{x \in {\Omega}\;:\;F(x)\preceq F(x^k)\;\mbox{ for all }\;k \in \mathbb{N}\}.
$$
From the $\mathbb{R}^m_+$-completeness assumption on $F$ at $x^0$, the set $E$ is nonempty. Take an arbitrary point $x^* \in E$, which means that $x^* \in \Omega_k$, and for all $ k \in \mathbb{N}$ denote $\gamma_{k+1}:=\lambda_k\langle \varepsilon^k,\alpha^{k+1}\rangle$. Observe that $\gamma_{k+1}>0$ whenever $k\in\mathbb{N}$ due to $\lambda_k>0$, $\varepsilon^k\in\mathbb{R}^{m}_{++}$ and $\alpha^k\in\mathbb{R}^{m}_{+}\backslash\{0\}$. Since 
$$
||x^k-x^*||^2 = ||x^k-x^{k+1}||^2+||x^{k+1}-x^*||^2+2\langle x^k-x^{k+1},x^{k+1}-x^*\rangle,
$$
we deduce from \eqref{condotim} that
\begin{equation}\label{teoconveqquasi1}
\begin{array}{ll}
\|x^k-x^*\|^2=\|x^k-x^{k+1}\|^2+\|x^{k+1}-x^*\|^2\\
+\displaystyle\dfrac{2}{\theta_{k+1}}\sum_{i\in \mathcal{I}}(\alpha_i^{k+1}+\displaystyle\beta_i^{k+1})\langle u_i^{k+1},x^{k+1}-x^*\rangle\\
+\dfrac{2}{\theta_{k+1}}\tau_{k+1}\langle \omega_i^{k+1},x^{k+1}-x^*\rangle,
\end{array}
\end{equation}
where $u_i^{k+1}\in\partial^M f_i(x^{k+1})$ for all $k\in\mathbb{N}$ and  $i\in \mathcal{I}$ with $\theta_{k+1}:=\lambda_{k+1}\sum_{i\in \mathcal{I}}\alpha_i^{k+1}\varepsilon_i^{k}$. On the other hand, it follows from the pseudoconvexity of $F$ together with $x^* \in \Omega_k$ and $\theta_k>0$ for all $ k$ that we have
\begin{equation}\label{eq1theoquasiconvex}
\frac{2}{\theta_{k+1}}\sum_{i=1}^{m}(\alpha_i^{k+1}+\beta_i^{k+1})\langle u_i^{k+1},x^{k+1}-x^*\rangle\geq 0.
\end{equation} 
The inclusion  $w^{k+1}\in N^M(x^{k+1};\Omega)$ combined  with $\tau_{k}>0$ leads us to
\begin{equation}\label{eq2theoquasiconvex}
\tau_{k+1}\langle w^{k+1},x^{k+1}-x^*\rangle\geq 0.
\end{equation} 
Using  now  \eqref{eq1theoquasiconvex} and \eqref{eq2theoquasiconvex} in \eqref{teoconveqquasi1} yields
\begin{equation}
\nonumber||x^{k+1}-x^k||^2\leq ||x^k-x^*||^2-||x^{k+1}-x^*||^2,\quad k\in \mathbb{N},
\end{equation}
which implies that $||x^{k+1}-x^*||\leq ||x^k-x^*||$ for any $x^*\in E$. This verifies that $\{x^k\}$ is Fej\'er convergent to $E$. \\
{\it Step}~2 ({\em weak cluster points of $\{x^k\}$ belongs to $E$}) Since $\{x^k\}$ is Fej\'er convergent to $E$, it follows from Proposition~\ref{fejer} that $\{x^k\}$ is bounded. Let $x^*$ be a weak cluster point of $\{x^{k}\}$. The construction of the  Algorithm tells us that $F(x^{k+1})\preceq F(x^{k})$ whenever $k\in\mathbb{N}$. Thus it follows from the continuity of $F$ that $F(x^*)\preceq F(x^{k})$ for all $k$, which verifies that $x^*\in E$.\\
{\it Step}~3 ({\em Convergence of the iterative sequence}) This step directly follows from Proposition~\ref{fejer} combined with Steps~1 and 2.\\
{\it Step}~4 ({\em Proximity of consecutive iterates}) Assume that $\{x^{k}\}$ converges to $\hat x$. It follows from the triangle inequality that
 \begin{equation}\label{trianginequality}
||x^{k+1}-x^k||\leq ||x^{k+1}-\hat{x}||+||x^k-\hat{x}||, \quad  k \in \mathbb{N},
\end{equation} 
due to  $x^k\to\hat{x}$ as $k\to\infty$, we conclude that
\begin{equation*}
\lim_{k\to\infty}||x^{k+1}-x^k||=0,
\end{equation*}
which justifies the claim of this step.\\
{\it Step}~5 ({\em The limit point is a weak Pareto solution}) The proof of this step uses the same argument as in the proof of Theorem~\ref{maintheorem} from \eqref{meanteoinq2} on. This completes the proof of the theorem. 
\end{proof}

\begin{acknowledgement}
Research of Boris S. Mordukhovich was partly supported by the US National Science Foundation under grant DMS-2204519, by the Australian Research Council under Discovery Project DP190100555, and by the Project 111 of China under grant D21024. Research of Glaydston de C. Bento was supported in part by CNPq grants 314106/2020-0 and the research of Jo\~ao Xavier da Cruz Neto was partly supported by CNPq grants 302156/2022-4.
\end{acknowledgement}

\end{document}